\documentclass{amsart}
\usepackage{}
\usepackage{amsfonts}
\usepackage{mathrsfs}
\usepackage{amsmath}
\usepackage{paralist}
\usepackage{amssymb}
\usepackage{amsthm}
\usepackage[colorlinks=true]{hyperref}
\hypersetup{urlcolor=blue, citecolor=red}

\numberwithin{equation}{section}
\newtheorem{Thm}{Theorem}[section]
\newtheorem{Lem}{Lemma}[section]
\newtheorem{Prop}{Proposition}[section]

\newtheorem{Cor}{Corollary}[section]
\newtheorem{example}{Example}[section]
\newtheorem{Rem}{Remark}[section]

\begin{document}
\title[ Infinitely Degenerate Mean Field Games]{Mean Field Games with infinitely degenerate diffusion and non-coercive Hamiltonian}
\thanks{Acknowledgements: This work is supported by the NSFC under the grands 12271269 and supported by the Fundamental Research Funds for the Central Universities.}
\author{Yiming Jiang}
\address{School of Mathematical Sciences and LPMC\\ Nankai University\\ Tianjin 300071 China}
\email{ymjiangnk@nankai.edu.cn}
\author{Jingchuang Ren}
\address{School of Mathematical Sciences \\ Nankai University\\ Tianjin 300071 China}
\email{1120200024@mail.nankai.edu.cn}
\author{Yawei Wei}
\address{School of Mathematical Sciences and LPMC\\ Nankai University\\ Tianjin 300071 China}
\email{weiyawei@nankai.edu.cn}
\author{Jie Xue}
\address{School of Mathematical Sciences \\ Nankai University\\ Tianjin 300071 China}
\email{1120200032@mail.nankai.edu.cn}
\begin{abstract}
In this paper, we consider a class of infinitely degenerate partial differential systems to obtain the Nash equilibria in the mean field games. The degeneracy in the diffusion and the Hamiltonian may be different. This feature brings difficulties to the uniform boundness of the solutions, which is central to the existence and regularity results. First, from the perspective of the value function in the stochastic optimal control problems, we prove the Lipschitz continuity and the semiconcavity for the solutions of the Hamilton-Jacobi equations (HJE). Then the existence of the weak solutions for the degenerate systems is obtained via a vanishing viscosity method. Furthermore, by constructing an auxiliary function, we conclude the regularity of the viscosity solution for the HJE in the almost everywhere sense.
\end{abstract}
\maketitle

\section{INTRODUCTION}
\subsection{Statement of the problem and motivation}
In this paper, we study the degenerate mean field game (briefly, MFG) systems as follows
\begin{equation}\label{1.2}
\left\{
  \begin{array}{ll}
    -\partial_t u-\mathcal {L}u+H(x,{D_G}u)=F(x,m)\qquad\quad\text{in}\ \mathbb{R}^{2}\times (0,T)\qquad\rm{(HJE)} \\
    \partial_t m-\mathcal {L}^*m-\mathrm{div}_G(mD_pH(x,D_Gu))=0\ \ \text{in}\ \mathbb{R}^{2}\times(0,T)\qquad\rm{(FPE)} \\
     u(x,T)=G(x,m_T),\  m(x,0)=m_0(x)\qquad\ x\in\mathbb{R}^{2}
  \end{array}
\right.
\end{equation}
where $\mathcal {L}$ is the second order operator given by
$$\mathcal {L}u(x,t):=\frac{1}{2}{\mathbf {tr}}\left(\sigma(x)\sigma'(x)D^2u(x,t)\right),$$
and $\mathcal {L}^*$ is the dual operator of $\mathcal {L}$. Here the diffusion matrix $\sigma(x)\sigma'(x)$ is possibly infinitely degenerate given by
\begin{equation}\label{asw}
  \sigma(x)=diag\{\sigma_1(x),\sigma_2(x)\}.
\end{equation}
For  any $x=(x_1,x_2)\in\mathbb{R}^{2}$, the Hamiltonian $H$ is non-coercive given by
$$H(x,D_Gu):=\frac{1}{2}|D_Gu|^2,$$
where $D_G$ is the gradient operator given by
\begin{equation}\label{G1}
D_Gu:=\big(\partial_{x_1}u,h(x_1)\partial_{x_2}u\big).
\end{equation}
Denote by $\mathrm{div}_Gu:=\partial_{x_1}u+h(x_1)\partial_{x_2}u$ the corresponding divergence operator. The function $h(x_1)$ is bounded and possibly infinitely vanishing. The functions $F$ and $G$ satisfy the assumptions {\bf{(H5)}} and {\bf{(H6)}} below. More precise assumptions are listed later.

To put the MFG systems \eqref{1.2} another way, let us consider a set of vector fields $\mathcal {X}=\{Y_1,Y_2\}$ of diagonal form,
where
 $$Y_i:=\sigma_i\partial_{x_i},\quad\text{for\ any\ }i\in\{1,2\}.$$
Then
 the MFG systems \eqref{1.2} can be regarded as the infinitely degenerate partial differential equations (briefly, PDE) systems induced by vector fields $\mathcal {X}$ as follows
\begin{equation}\label{1.22}
\begin{cases}
    -\partial_t u-{\sum^2_{i=1}}(\frac{1}{2}Y^2_iu- Y_iu\mathrm{div}Y_i)+\frac{1}{2}|D_Gu|^2=F(x,m),&\text{in}\ \mathbb{R}^{2}\times (0,T), \\
       \partial_t m-{\sum^2_{i=1}}\left(\frac{1}{2}(Y^*_i)^2m-Y^*_i(m\mathrm{div}Y_i)\right)-\mathrm{div}_G(mD_Gu)=0, &\text{in}\ \mathbb{R}^{2}\times(0,T), \\
     u(x,T)=G(x,m_T),\  m(x,0)=m_0(x),\qquad\qquad\qquad\quad\  &x\in \mathbb{R}^{2},
\end{cases}
\end{equation}
where $Y_i^*=-Y_i-\mathrm{div} Y_i$ is the dual operator of $Y_i$, and $\mathrm{div} Y_i=\partial_{x_i}\sigma_i$ is the divergence of the vector fields $Y_i$.

%

The motivation of the PDE systems \eqref{1.2} is to describe Nash equilibria in the following MFG. The $u$ in (HJE) is the value  function of an optimal control problem of a generic player, where the dynamics is given by the controlled stochastic differential equations (briefly, SDE)
\begin{align}\label{eq1.7}
\left\{
 \begin{array}{ll}
  dX_{1,s}=\alpha_{1,s}ds+\sigma_1(X_s)dB_{1,s},  \\
  dX_{2,s}=\alpha_{2,s}h(X_{1,s})ds+\sigma_2(X_s)dB_{2,s},\\
   X_{1,t}=x_1\in\mathbb{R},\  X_{2,t}=x_2\in\mathbb{R}.\\
\end{array}
\right.
\end{align}
For any $s\in[t,T]$, set
$X_s:=(X_{1,s},X_{2,s})$, $\alpha_s:=(\alpha_{1,s},\alpha_{2,s})$, $B_s:=(B_{1,s},B_{2,s})$, $b(X_s,\alpha_s):=(\alpha_{1,s}, \alpha_{2,s}h(X_{1,s}))$ and $\sigma(X_s):=diag\big\{\sigma_{1}(X_s), \sigma_2(X_s)\big\}$.

Then if the evolution of the whole population's distribution $m$ is given, each player wants to choose the optimal control $\alpha_s$ to minimize the cost function
\begin{equation}\label{4.21}
u(x,t)=\inf_{\alpha\in\mathscr{A}(x,t)}\mathbf{E}\bigg[\int^T_t\frac{1}{2}|\alpha_s|^{2}+F(X_s,m_s)ds+G(X_T,m_T)\bigg],
\end{equation}
where $\mathscr{A}(x,t)$ is the set of control processes $\alpha$ such that
\begin{equation}\label{2.28}
\mathbf{E}\bigg[\int^T_tF(X_s,m_s)ds\bigg]<\infty,\ \text{and}\ \mathbf{E}\bigg[\int^T_t|\alpha_s|^2ds\bigg]<\infty.
\end{equation}
In the SDE \eqref{eq1.7}, the drift coefficient $b(X_{s},\alpha_s)$ and the diffusion coefficient $\sigma(X_s)$ are Lipschitz w.r.t. $X_s$ uniformly in $\mathscr{A}(x,t)$. Assume that $B_{1,s}$ and $B_{2,s}$  are independent one dimension standard Brownian motion on a filtered probability space $(\mathbf{\Omega},\mathscr{F},\{\mathscr{F}_t\}_{t\geq0},\mathbf{P})$ satisfying the usual conditions in the stochastic analysis, see Chapter 3 in \cite{Hu}.
 Note that controls $\alpha_s$ are adapted to the filtration generated by $B_s$, valued in $\mathscr{A}(x,t)$. The optimal feedback of each player is given by
\begin{equation}\label{11.20}
\alpha^*(x,t)=-D_{p}H(x,D_G u)=-D_G u(x,t).
\end{equation}

We emphasize that the degeneracy of $Y_i$ caused by the diffusion term and $D_G$ caused by the drift term in the SDE \eqref{eq1.7} is inconsistent.

Now we first give a derivation of the Hamiltonian-Jacobi equations (briefly, HJE)
 \begin{equation}\label{5.2}
\left\{
  \begin{array}{ll}
     -\partial_t u-\mathcal {L}u+H(x,{D_G}u)=F(x,\bar{m}),\ \ \quad\ \text{in}\ \mathbb{R}^{2}\times (0,T), \\
    u(x,T)=G(x,\bar{m}_T),\qquad\qquad\qquad\qquad\quad x\in\mathbb{R}^{2},
  \end{array}
\right.
\end{equation}
with the fixed measure $\bar{m}$, which refers to Chapter 2 in \cite{LR}. 
 For any stopping time $\tau\in[0,T]$, using the It\^{o}'s formula to $u$ on $[t,\tau]$ and combined with the SDE \eqref{eq1.7}, we have
\begin{align}\label{2.2}
u(X_\tau,\tau)-u(x,t)
  &= \int^\tau_t \partial_s u(X_s,s)ds+\int^\tau_tDu(X_s,s)dX_s+\frac{1}{2}\int^\tau_tD^2u(X_s,s)d\big<X\big>_s\\
   &= \int^\tau_t\left(\partial_s+\mathcal {L}+\alpha_s  D_G\right)u(X_s,s)ds+\int^\tau_tDu(X_s,s)\sigma(X_s)dB_s.
\end{align}
By the dynamic programming principle (briefly, DPP), we have
\begin{equation}\label{1.08}
 u(x,t)=\inf_{\alpha\in\mathscr{A}(x,t)}\mathbf{E}\left[\int^\tau_t \frac{1}{2}|\alpha_s|^{2}+F(X_s,\bar{m}_s)ds+u(X_\tau,\tau)\right].
 \end{equation}
By the martingale property, we have $\mathbf{E}\left[\int^\tau_t Du(X_s,s)\sigma(X_s)dB_s\right]=0.$ Plugging \eqref{2.2} into \eqref{1.08}, we have
\begin{align*}
   \inf_{\alpha\in\mathscr{A}(x,t)}\mathbf{E}\left[\int^{\tau}_{t}\frac{1}{2}|\alpha_s|^2+F(X_s,\bar{m}_s)ds+\int^{\tau}_{t}(\partial_s+\mathcal {L} +\alpha_s  D_G) u(x,s)ds\right] &=0.
\end{align*}
Let $\tau=t+\delta$, divide by $\delta$ and let $\delta\rightarrow0$, we obtain the HJE
\begin{equation}\label{1.7}
\partial_tu(x,t)+\mathcal {L}u(x,t)+F(x,m)+\inf_{\alpha_t\in\mathscr{A}(x,t)}\big\{H^*(x,\alpha_t)+\alpha_t  D_G u(x,t)\big\}=0,
\end{equation}
where $H^*(x,\alpha_s):=\frac{1}{2}|\alpha_s|^{2},$ is the Fenchel conjugate of the Hamiltonian $H(x,D_Gu)$ w.r.t. the second variable. Hence the HJE \eqref{5.2} is valid. 

Next, 
we give the derivation of the Fokker-Planck equation (briefly, FPE)
\begin{equation}\label{5.4}
\left\{
  \begin{array}{ll}
    \partial_t m-\mathcal {L}^*m-D_Gm  D_Gu-m\Delta_Gu=0,\quad\ \ \text{in}\ \mathbb{R}^{2}\times(0,T),\\
     m(x,0)=m_0(x),\qquad\qquad\qquad\qquad\qquad\quad\ x\in\mathbb{R}^{2},
  \end{array}
\right.
\end{equation}
where $\Delta_G u:=\partial^2_{x_1}u+h^2(x_1)\partial^2_{x_2}u$, which refers to Chapter 1 in \cite{Carmona16}.
If $\varphi\in C^\infty_0(\mathbb{R}^2)$ is any test function, as \eqref{2.2}, then the It\^{o}'s formula gives
\begin{align}\label{02.3}
  \varphi(X_t) 
   &= \varphi(X_0)+\int^{t}_{0}\big(\mathcal {L}+D_Gu  D_G \big)\varphi (X_s)ds+\int^{t}_{0}D\varphi(X_s)\sigma(X_s) dB_s.
\end{align}
Denote the distribution of $X_t$ by $\mu_t(dx)=\mathbf{P}(X_t\in dx)$, and use the notation
$\left<\varphi,\mu\right>:=\int_{\mathbb{R}^2}\varphi(x)\mu(dx).$
Taking expectations on both sides of \eqref{02.3}, we have
$$\left <\varphi,\mu_t\right> =\big <\varphi,\mu_0+\int^{t}_{0}\mathcal {L}^*\mu_s +\mathrm{div}_G(\mu_s D_Gu) ds\big>.$$
Assume that $\mu_t$  has a density satisfying $\mu_t(dx)=m(x,t)dx$. For the arbitrary of $\varphi$, then the density $m$ is the solution of
\begin{equation}\label{eq2.11}
\partial_t m=\mathcal {L}^* m+\mathrm{div}_G(m D_Gu)=\mathcal {L}^* m+D_Gm  D_Gu+m\Delta_Gu,
\end{equation}
with initial distribution $m_0$, and the FPE \eqref{5.4} is valid.
\subsection{Research history and main results}
MFG theory is devoted to the analysis of differential games with infinitely many players. This theory has been introduced by Lasry and Lions \cite{1} and \cite{2}. At about the same time, Huang et al. \cite{3} solved the large population games independently. Then MFG has been studied extensively in many different fields. Bensoussan et al. \cite{Be} studied the MFG and mean field type control theory. Carmona and Delarue \cite{6} focused on the theory and applications of MFG by probabilistic approach. Gangbo \cite{5} developed optimal transport theory within the MFG framework. Gomes et al. \cite{Gomes} discussed regularity theory for MFG systems either stationary or time-dependent, local or nonlocal. Cardaliaguet et al. \cite{Cardaliaguet15} obtained  the existence of classical solutions for the master equation of MFG. The notes written by Cardaliaguet \cite{Cardaliaguet} and by Ryzhik \cite{LR} showed the more comprehensive analysis of the MFG.

Degenerate MFG systems have much fewer references than the classical ones. Cardaliaguet et al. \cite{11} tackled the degenerate second order systems with local coupling and coercive first order operators, and they established the existence and uniqueness of suitably defined weak solutions by using a variational approach. Mannucci et al. \cite{0} studied the non-coercive first order MFG systems and obtained a weak solution via a vanishing viscosity method. Ferreira et al. \cite{80} obtained the existence of weak solutions to a wide class of time-dependent monotone MFG by using Minty's method. Cardaliaguet et al. \cite{Com} developed a new notion of weak solutions for the MFG with common noise and degenerate idiosyncratic noise. Under the H\"{o}rmander condition introduced by H\"{o}rmander \cite{Hormander}, Dragoni and Feleqi \cite{16} studied the second order ergodic systems, which enjoy more regularily properties than general degenerate systems. 
In this paper, the degeneracy of the second order MFG systems \eqref{1.2} in the diffusion term and the Hamiltonian may be different and without H\"{o}rmander condition.

The infinitely degenerate operators do not satisfy the H\"{o}rmander condition, which brings difficulties for regularity. 
Melrose and Mendoza \cite{012} studied the infinitely degenerate elliptic operators and the mapping properties of elliptically totally characteristic pseudo differential operators by the so-called B-calculus. Schulze and his group worked on the microlocal analysis of the infinitely degenerate elliptic operators arising from manifolds with singularities, see \cite{018} \cite{019} and the references therein. Morimoto and Xu \cite{33} studied the semilinear Dirichlet problems of infinitely degenerate operators and proved the existence and regularity of weak solutions under the assumption of logarithmic regularity estimates. Chen et al. \cite{CW3}-\cite{CW2} focused on the infinitely degenerate operators, and gave the existence of the distribution solution for the  semilinear  degenerate elliptic equations. Here the MFG systems \eqref{1.2} can be regarded as the PDE systems \eqref{1.22} with the infinitely degenerate operators.


The references mentioned above motivate us to discover the existence and uniqueness of the coupling solutions for the degenerate MFG systems \eqref{1.2}.
Now we list our notions and assumptions as follows.

 Let $\mathcal {P}_{1}$  be the set of Borel probability measures $m$ on $\mathbb{R}^{2}$, such that $\int_{\mathbb{R}^{2}}|x|dm_t(x)<\infty$, and  endowed with the Kantorovitch-Rubinstein distance
$$d_1(\mu,\nu):=\inf_{\gamma\in\Pi(\mu,\nu)}\int_{\mathbb{R}^2}|x-y|d\gamma(x,y),$$
where $\Pi(\mu,\nu)$ is the set of Borel probability measures on $\mathbb{R}^2$ such that $\gamma(E\times\mathbb{R}^2)=\mu(E)$,
 and $\gamma(\mathbb{R}^2\times E)=\nu(E)$ for any Borel set $E\subset\mathbb{R}^2$. More details of this distance can refer to  \cite{Cardaliaguet}.

Let $\mathcal {C}$ be the set of maps $\mu\in C([0,T],\mathcal {P}_{1})$, such that
 $\sup\limits_{t\in[0,T]}\int_{\mathbb{R}^{2}}|x|^2dm_t\leq C$ and
\begin{equation}\label{eq2.7}
\sup_{s\neq t}\frac{d_1(\mu_s,\mu_t)}{|t-s|^{\frac{1}{2}}}\leq C.
\end{equation}
Then $\mathcal {C}$ is a convex closed subset of $C([0,T];\mathcal {P}_{1})$ and compact in $\mathcal {P}_1$, more properties of the space of probability measures refer to Chapter 5 in \cite{Cardaliaguet}.

Denote $C^2(\Omega)$ the space of functions with continuous second order derivatives endowed with the norm $$\|f\|_{C^2(\Omega)}:=\sup_{x\in\Omega}\big\{|f(x)|+|Df(x)|+|D^2f(x)|\big\}.$$

Throughout this paper, $C>0$ is a generic constant which may differ from line by line, and the following assumptions are required.

\noindent{\bf(H1)} The function $h:\mathbb{R}\to\mathbb{R}$ is possibly infinitely vanishing. Moreover, the set  $\mathcal {N}(h):=\{y\in\mathbb{R}|h(y)=0\}\neq\emptyset$.

\noindent\textbf{(H2)} The function $h$ is $C^2(\mathbb{R})$ with $\|h\|_{C^2(\mathbb{R})}\leq C$.

\noindent\textbf{(H3)} The entries of the diffusion matrix $\sigma(x)$ are Lipschitz continuous and $C^2(\mathbb{R}^{2})$ with
$\|\sigma_{i}(\cdot)\|_{C^2(\mathbb{R}^{2})}\leq C$, for any $i\in\{1,2\}$.

\noindent\textbf{(H4)} The initial distribution $m_0$ is absolutely continuous w.r.t. the Lebesgue measure and has a $ C^{2,\alpha}$ continuous density, for $\alpha\in(0,1)$, still denoted by $m_0$, which satisfies $\int_{\mathbb{R}^2}|x|^2 dm_0<\infty$.

\noindent\textbf{(H5)} The functions $F(x,m)$ and $G(x,m_T)$ are real-valued continuous functions on $\mathbb{R}^{2} \times\mathcal{P}_1$ and Lipschitz continuous from $\mathcal{P}_1$ to $C^2(\mathbb{R}^{2})$ uniformly for $x\in\mathbb{R}^{2}$. Moreover, there exists a constant $C>0$, such that for any $m\in\mathcal{P}_1$
$$\|F(\cdot,m)\|_{C^2(\mathbb{R}^{2})}+\|G(\cdot,m_T)\|_{C^2(\mathbb{R}^{2})}\leq C.$$

\noindent\textbf{(H6)} The $F(\cdot,m)$ and $G(\cdot,m_T)$ are monotonically increasing w.r.t. the measure $m$.

The higher regularity of the solution for the HJE \eqref{1.2} requires the following assumption.

\noindent\textbf{(H7)} There exists a constant $C>0$, for any $m\in\mathcal{P}_1$ such that
$$\|F(\cdot,m)\|_{C^3(\mathbb{R}^{2})}+\|G(\cdot,m_T)\|_{C^3(\mathbb{R}^{2})} +\|h\|_{C^3(\mathbb{R})}+\|\sigma_{i}(\cdot)\|_{C^3(\mathbb{R}^{2})}\leq C.$$

\begin{example}
For any $(x_1,x_2)\in\mathbb{R}^2$, a special family of vector fields is given by
\begin{equation}\label{1.002}
 \mathcal {X}=\{Y_1,Y_2\},\ \text{with}\ \ Y_1=\partial_{x_1}, Y_2=\lambda(x_1)\partial_{x_2},
\end{equation}
where $\lambda$ is a continuous function and is expected for at most a finite number of zero points.
This kind of vector fields $\mathcal {X}$ is Grushin type, which has been studied in many works \cite{41} and \cite{48}.

Specifically, take the function
\begin{equation}\label{1.03}
\lambda(x)=\begin{cases}
e^{-\frac{1}{x^2}}, &\text{if}\  x\neq0, \\
0,                  &\text{if}\  x=0.
\end{cases}
\end{equation}
Then $\lambda(x)$ vanishes to any order at the origin. Hence the family of vector fields $\mathcal {X}$ in \eqref{1.002} is infinitely degenerate.

Note that $D_Gu$ defined in \eqref{G1} has the Grushin structure, then the example \eqref{1.03} can also adapt to the function $h(x_1)$.
\end{example}
\begin{example}
Easy examples for a family of infinitely degenerate vector fields $\mathcal {X}=\{Y_1,Y_2\}$ satisfying the assumption {\bf{(H3)}} are given by
$$Y_1=sinx_1\partial_{x_1},\ \text{or}\ \frac{x_1}{\sqrt{x_1+1}}\partial_{x_1}, \  \text{or}\ \frac{x_1}{e^{x_1}}\partial_{x_1},\ \text{and}\ Y_2=\partial_{x_2}.$$
\end{example}

Here below we state the main results of this paper.
First, we discuss the regularity of the value function, which are used later to prove the existence of weak solutions to the MFG systems \eqref{1.2}.
\begin{Lem}\label{Lem2.2}(Lipschitz continuity)
Under assumptions {\bf{(H1)}}-{\bf{(H5)}}, then the HJE \eqref{5.2} has a unique bounded uniformly continuous viscosity solution given by \eqref{4.21}. Moreover,
 the value function $u(x,t)$ defined in \eqref{4.21} is Lipschitz continuous w.r.t. the spatial variable $x$ and the time variable$t$ respectively.
\end{Lem}
\begin{Lem}\label{2.3}(Semiconcavity)
Under assumptions {\bf{(H1)}}-{\bf{(H5)}}, the value function $u(x,t)$ defined in \eqref{4.21} is semiconcave w.r.t. the variable $x$.
\end{Lem}
Next, fixed the measure $\bar{m}\in C([0,T];\mathcal {P}_{1})$, we construct the auxiliary systems
\begin{equation}\label{4.2}
\left\{
  \begin{array}{ll}
    -\partial_t u-\big(\epsilon\Delta+\mathcal {L}\big) u+H(x,{D_G}u)=F(x,\bar{m}),\qquad\ \text{in}\ \mathbb{R}^{2}\times(0,T), \\
     \partial_t m-\big(\epsilon\Delta+\mathcal {L}^*\big) m-\mathrm{div}_G(mD_Gu)=0,\qquad\qquad\text{in}\ \mathbb{R}^{2}\times(0,T), \\
     u(x,T)=G(x,\bar{m}_T),\  m(x,0)=m_0,\qquad\qquad\qquad\ x\in\mathbb{R}^{2},
  \end{array}
\right.
\end{equation}
and obtain the existence and uniqueness of the vanishing viscosity limit as follows.
\begin{Prop}\label{3.1}
Under assumptions {\bf{(H1)-(H5)}}, let $(u^\epsilon,m^\epsilon)$ be a unique coupling of the solutions to the auxiliary systems \eqref{4.2}. Then up to a subsequence, $\{u^\epsilon\}$ converges to $u$ in $\mathbb{R}^2\times[0,T]$, and $\{m^\epsilon\}$ converges to $m$ in $C([0,T],\mathcal {P}_1)$, where $u$ is a solution for the HJE \eqref{5.2} in the viscosity sense, and $m$ is a solution for the FPE \eqref{5.4} in the sense of distribution.
\end{Prop}
Then, by the monotonicity condition \textbf{(H6)} and Schauder fixed point theorem, we obtain the main result.
\begin{Thm}\label{2.4}
Under assumptions {\bf{(H1)-(H6)}}, there exists a unique coupling of the solutions $(u,m)\in W^{2,\infty}(\mathbb{R}^{2}\times[0,T])\times C([0,T],\mathcal {P}_1)$ of the infinitely degenerate MFG systems \eqref{1.2}, where $u$ is in the viscosity sense and $m$ is in the sense of distributions.
\end{Thm}

Finally, adding the assumption {\bf(H7)}, we obtain the higher regularity result.
\begin{Thm}\label{2.44}
Under assumptions {\bf{(H1)-(H7)}}, the viscosity solution $u$ in the MFG systems \eqref{1.2} satisfies the HJE \eqref{5.2} a.e..
\end{Thm}

\begin{Rem}
For the infinitely degenerate vector fields $\mathcal {X}=\{Y_1,Y_2\}$, we introduce the following function space
$$H^1_{\mathcal {X}}(\Omega)=\left\{u\in L^2(\Omega)\big|Y_iu\in L^2(\Omega),\ i=1,2\right\},$$
which is a Hilbert space with norm $$\|u\|^2_{H^1_{\mathcal {X}}(\Omega)}=\|u\|^2_{L^2(\Omega)}+\sum_{i=1,2}\|Y_iu\|^2_{L^2(\Omega)}.$$
Then for any $k> 2$, the embedding $H^1_{\mathcal {X}}(\Omega)\hookrightarrow L^k(\Omega)$ will not hold. In fact, for the infinitely degenerate vector
fields $\mathcal {X}$, the critical index of the embedding mapping is at most $2$. In other word, one cannot expect the solutions in the system \eqref{1.2} to possess any additional regularity. Hence the results given in Theorem \ref{2.4} and Theorem \ref{2.44} are reasonable.
\end{Rem}

\begin{Rem}\label{H}
Generally, we extend the MFG systems \eqref{1.2} to a high dimensional case. For any $x=(x_1,x_2)$, $x_1\in\mathbb{R}^{m},\ x_2\in\mathbb{R}^{n}$, and $m+n=N\in \mathbb{N}$. Let $h:\mathbb{R}^m\to\mathbb{R}$ be a regular function, $\sigma(X_s)=diag\{I_m,h(X_{1,s})I_n\}$, $B_{1,s}$ and $B_{2,s}$ be $m$-, $n$-dimension independent standard Brownian motions. Then we can generalize our results to $N$-dimension.
\end{Rem}
\begin{Rem}
If the diffusion in the SDE \eqref{eq1.7} has the following Grushin structure
$$\sigma(X_s)=diag\big\{1,h(X_{1,s})\big\},$$
then the MFG systems \eqref{1.22} can be rewritten as
\begin{equation}\label{31}
\left\{
  \begin{array}{ll}
    -\partial_t u-\frac{1}{2}\sum^2_{i=1}Y^2_iu+\frac{1}{2}|Yu|^{2}=F(x,m), \quad\ \ \ \text{in}\ \mathbb{R}^{2}\times (0,T), \\
    \partial_t m-\frac{1}{2}\sum^2_{i=1}Y^2_im-\mathrm{div}_\mathcal {X}(mYu)=0,\quad\qquad\text{in}\ \mathbb{R}^{2}\times(0,T), \\
     u(x,T)=G(x,m_T),\  m(x,0)=m_0(x),\qquad\quad\ x\in \mathbb{R}^{2},
  \end{array}
\right.
\end{equation}
 where $Yu:=(Y_1u,Y_2u)$ is the gradient associated to the vector fields $\mathcal {X}=\{Y_1,Y_2\}$ and $\mathrm{div}_{\mathcal {X}}u:=Y_1u+Y_2u$ is the corresponding divergence operator.
Here, from the perspective of a single player, when $h$ is vanishing, there may be a ``forbidden'' direction.

 Furthermore, if the vector fields $\mathcal {X}$ satisfies H\"{o}rmander condition \textup{(}see details in \cite{Hormander}\textup{)}, then there exists a unique coupling of classical solutions in the weighted H\"{o}lder space. This work has been studied in \cite{DMFG1}.
\end{Rem}
\begin{Rem}
The main results given in Theorem \ref{2.4} and Theorem \ref{2.44} are also valid for the general non-diagonal matrix $\sigma(x,t)$. Here the diagonal assumption of $\sigma$ in \eqref{asw} is purely intended to characterize the optimal control problems by the SDE \eqref{eq1.7}. The proof of the theorems does not depend on it.
\end{Rem}

The main contributions of this paper are summarized in the following three points. First, we prove the existence and uniqueness of coupling solutions for the degenerate PDE systems \eqref{1.2} in Theorem \ref{2.4}, which describes the Nash equilibria in the MFG. The infinite degeneracy in the diffusion and the Hamiltonian may be different, which complicates the proof of the existence and regularity results. This feature distinguishes the present work from the existing references on the MFG systems and provides a more flexible framework for modeling complex reality.
Second, for the HJE in the systems \eqref{1.2}, the global semiconcavity plays a critical role in the existence result.
Since the HJE is satisfied by the value function \eqref{4.21} of the stochastic optimal control problems, in Lemma \ref{2.3} we prove the semiconcavity for the value function by a new method combining stochastic analysis with the computation in PDE.
Third, we conclude the regularity of the viscosity solution for the HJE in the almost everywhere sense in Theorem \ref{2.44}. In the vanishing viscosity method, the degeneracy brings difficulties to the uniform boundness of third-order derivatives of the solutions for the approximate problems. Here we construct an auxiliary function comprised by a linear combination with derivatives of various orders to overcome it.

The rest of the paper is organized as follows. In Section 2, we prove Lemma \ref{Lem2.2} and Lemma \ref{2.3} to obtain the Lipschitz continuity and semiconcavity of the value function. In Section 3, we show the existence and uniqueness of the vanishing viscosity limit given in Proposition \ref{3.1}. In Section 4, we give the proof of the main results Theorem \ref{2.4} and Theorem \ref{2.44} to obtain the existence and uniqueness of the MFG systems \eqref{1.2} and the higher regularity result.

\section{THE REGULARITY OF THE VALUE FUNCTION}
In Subsection 2.1, we give some known results in PDE and stochastic analysis. In Subsection 2.2, we prove Lemma \ref{Lem2.2} to obain the Lipschitz continuity of the value function. In Subsection 2.3, we prove the semiconcavity of the value function given in Lemma \ref{2.3}.
\subsection{Preliminaries and known results}

First, we give the property concerning semiconcave functions.


\begin{Lem}\label{2.1}{\bf{(Theorem 3.3.3 in \cite{Semi})}}
Let $A \subset \mathbb{R}^m$ be an open set and $\{u_n\}:A\rightarrow\mathbb{R}$ be a family of semiconcave functions with the same modulus. Given an open set $B\subset \subset A$, suppose that the $u_n$'s are uniformly bounded in $B$. Then there exists a subsequence $\{u_{n_k}\}$ converging uniformly to a function $u:B\rightarrow\mathbb{R}$ semiconcave with same modulus. In addition, $Du_{n_k}\rightarrow Du$ a.e. in $B$.
\end{Lem}

Next, we have the following isometry relation w.r.t. the stochastic integral.
\begin{Lem}\label{Norm}{\bf{(Theorem 4.2(c) in \cite{06})}}
Assume that $Z$ is bounded continuous martingales. Let $\Pi_2(Z)$ be the set of all predictable processes $K$ that have $\mathbf{E}\left[\int^T_t K_\tau^2d\left<Z\right>_\tau\right]<\infty$. If $K\in\Pi_2(Z)$, then
$$\mathbf{E}\bigg[\bigg(\int^T_t K_\tau dZ_\tau\bigg)^2\bigg]=\mathbf{E}\bigg[\int^T_t K_\tau^2d\left<Z\right>_\tau\bigg].$$
Furthermore, in the case of a Brownian motion, since $\left<B\right>_\tau=\tau$, we have
$$\mathbf{E}\bigg[\bigg(\int^T_t K_\tau dB_\tau\bigg)^2\bigg]=\mathbf{E}\bigg[\int^T_t K_\tau^2d\tau\bigg].$$
\end{Lem}
Then, we give the estimates on the moments of solutions to the SDE \eqref{eq1.7}.
\begin{Lem}\label{SDE}{\bf{(Theorem 1.3.16 in \cite{0})}}
For the SDE \eqref{eq1.7}, if $b$ and $\sigma$  are Lipschitz with linear growth conditions, such that for some constant $\beta$
 $$\big(b(x,t,\omega)-b(y,t,\omega)\big)(x-y)+\frac{1}{2}\big|\sigma(x,t,\omega)-\sigma(y,t,\omega)\big|^2\leq\beta|x-y|^2.$$
 Then there exists a unique strong solution, and for all $0\leq t\leq s\leq T$, $x,y\in\mathbb{R}^2$,
\begin{equation}\label{9.1}
  \mathbf{E}\left[\sup_{t\leq\tau\leq s}\big|X_{\tau}^{t,x}-X_{\tau}^{t,y}\big|^2\right]\leq e^{2\beta(s-t)}|x-y|^2.
\end{equation}
\end{Lem}

\begin{Cor}\label{SD}
Under same assumptions in Lemma \ref{SDE}, for any $s\in[t,T]$ we have
\begin{equation}\label{1.02}
\mathbf{E}\left[\sup_{t\leq\tau\leq s}\big|X_{\tau}^{x,t}-X_{\tau}^{y,t}\big|^4\right]\leq C|x-y|^4,
\end{equation}
and\begin{equation}\label{1.01}
\mathbf{E}\left[\sup_{t\leq\tau\leq s}\big|X_{\tau}^{x,t}-X_{\tau}^{y,t}\big|\right]\leq C|x-y|,
\end{equation}
where $C$ is a positive constant depending on $s$, $t$ and $\beta$ in Lemma \ref{SDE}.
\end{Cor}
\begin{proof}
First, we prove the estimate \eqref{1.02}. It follows from SDE \eqref{eq1.7} that
$$X_{\tau}^{y,t}-X_{\tau}^{x,t}=y-x+\int^\tau_tb(X_{r}^{y,t})-b(X_{r}^{x,t})dr+\int^\tau_t\sigma(X_{r}^{y,t})-\sigma(X_{r}^{x,t})dB_r.$$
Since $(a+b)^4\leq8a^4+8b^4$, we have
\begin{align*}
   \mathbf{E}\left[\sup_{t\leq\tau\leq s}\big|X_{\tau}^{y,t}-X_{\tau}^{x,t}\big|^4\right] &\leq C\mathbf{E}\bigg[|y-x|^4+\sup_{t\leq\tau\leq s}\Big|\int^\tau_tb(X_{r}^{y,t})-b(X_{r}^{x,t})dr\Big|^4\\
   &\quad+\sup_{t\leq\tau\leq s}\Big|\int^\tau_t\sigma(X_{r}^{y,t})-\sigma(X_{r}^{x,t})dB_r\Big|^4\bigg].
\end{align*}
It follows from 
 Lemma \ref{Norm} and Doob's inequality that
\begin{align*}
   &\quad\mathbf{E}\left[\sup_{t\leq\tau\leq s}\big|X_{\tau}^{y,t}-X_{\tau}^{x,t}\big|^4\right]\\
    &\leq C\mathbf{E}\bigg[|y-x|^4+\int^s_t\left|b(X_{r}^{y,t})-b(X_{r}^{x,t})\right|^4dr\ +\int^s_t\left|\sigma(X_{r}^{y,t})-\sigma(X_{r}^{x,t})\right|^4dr\bigg].
\end{align*}
By the SDE \eqref{eq1.7} and Lipschitz continuity of $b$ and $\sigma$ given in {\bf{(H2) (H3)}}, we have
\begin{align*}
  \mathbf{E}\left[\sup_{t\leq\tau\leq s}\big|X_{\tau}^{y,t}-X_{\tau}^{x,t}\big|^4\right] &\leq C|y-x|^4+C\mathbf{E}\left[\int^s_t\left|X_{r}^{y,t}-X_{r}^{x,t}\right|^4dr\right]\\
  &\leq C|y-x|^4+C\int^s_t\mathbf{E}\left[\sup_{t\leq\tau\leq r}\left|X_{\tau}^{y,t}-X_{\tau}^{x,t}\right|^4\right]dr.
\end{align*}
Then Gronwall's inequality gives that
\begin{align*}
 \mathbf{E}\left[\sup_{t\leq\tau\leq s}\big|X_{\tau}^{y,t}-X_{\tau}^{x,t}\big|^4\right]
  &\leq C|y-x|^4+C\int^s_te^{C(s-t)}|y-x|^4ds \leq C_{\beta,s,t}|y-x|^4,
\end{align*}
and the inequality \eqref{1.02} follows.

To prove the estimate \eqref{1.01}, using the Jensen inequality and \eqref{1.02}, we have
\begin{align*}
\mathbf{E}\Big[\sup_{t\leq\tau\leq s}\big|X_{\tau}^{x,t}-X_{\tau}^{y,t}\big|\Big]&\leq \left(\mathbf{E}\left[\sup_{t\leq\tau\leq s}\big|X_{\tau}^{x,t}-X_{\tau}^{y,t}\big|^4\right]\right)^{\frac{1}{4}}\leq C|x-y|.
\end{align*}
Then the result follows.
\end{proof}

\subsection{Proof of Lemma \ref{Lem2.2}}
\begin{proof}[{\bf{Proof of Lemma \ref{Lem2.2}}}]
\noindent{\bf{Step\ 1.}}
First, let us recall that there exists a unique viscosity solution $u$ of HJE \eqref{5.2}, which refers to Theorem 4.3.1 and Theorem 4.4.5 in \cite{Hu}. Moreover, the dynamic programming principle is satisfied by the value function \eqref{4.21}. It is clearly that $u$ is a bounded and uniformly continuous by the assumption {\bf{(H5)}}.

\noindent{\bf{Step\ 2.}}
Then, we prove the Lipschitz continuity of $u(x,t)$ w.r.t. $x$, which refers to Lemma 4.7 in \cite{Cardaliaguet}. Let $t$ be fixed, and $\alpha^\varepsilon$ be an $\varepsilon$-optimal control for $u(x,t)$, i.e.
\begin{equation}\label{2.6}
  u(x,t)+\varepsilon\geq \mathbf{E}\bigg[\int^T_t\frac{1}{2}|\alpha^\varepsilon_s|^{2}+F(X^{x,t}_s,m_s)ds+G(X^{x,t}_T,m_T)\bigg],
\end{equation}
where $X^{x,t}_s$ obeys to the SDE \eqref{eq1.7} with the control $\alpha^\varepsilon$. 

We consider the path $X^{y,t}_s$ starting from $y=(y_1,y_2)\in\mathbb{R}^{2}$ with  the control $\alpha^\varepsilon$.
For the sake of brevity, we write $F(X_s):=F(X_s,m_s)$, and $G(X_T):=G(X_T,m_T)$. It follows from the assumption {\bf{(H5)}} and Corollary \ref{SD} that
\begin{align*}
  \mathbf{E}\bigg[\int^T_t F(X^{y,t}_s)-F(X^{x,t}_s)ds\bigg] & \leq C\mathbf{E}\bigg[\int^T_t\sup_{t\leq s\leq T} \big|X^{y,t}_s-X^{x,t}_s\big|ds\bigg] \leq C|y-x|.
\end{align*}
By the same calculations for $G$, and substituting inequality \eqref{2.6} in
\begin{equation}\label{2.01}
  u(y,t) \leq \mathbf{E}\bigg[\int^T_t\frac{1}{2}|\alpha^\varepsilon_s|^{2}+F(X^{y,t}_s)ds+G(X^{y,t}_T)\bigg],
\end{equation}
then we get
\begin{align*}
 u(y,t) 
   &\leq u(x,t)+\varepsilon+C\mathbf{E}\bigg[\int^{T}_{t}\big|F(X^{y,t}_s)-F(X^{x,t}_s)\big|ds+\big|G(X ^{y,t}_T)-G(X^{x,t}_T)\big|\bigg]\\
  &\leq u(x,t)+\varepsilon+C|y-x|.
\end{align*}
Reversing the role of $x$ and $y$, for the arbitrary of $\varepsilon$, the Lipschitz continuity w.r.t. $x$ holds.

Moreover, since $u$ is Lipschitz continuous w.r.t $x$, the optimal control $\alpha^*_s$ given in \eqref{11.20} is bounded for any $s\in[t,T]$. Furthermore, we have $\alpha^*\in\mathcal {A}(x,t)$ and
\begin{equation}\label{2.25}
|\alpha^*|_{\infty}:=\sup_{t\leq s\leq T}\sup_{\omega\in\Omega}|\alpha^*_s(\omega)|\leq C.
\end{equation}

\noindent{\bf{Step\ 3.}}
Finally, we prove the Lipschitz continuity of $u(x,t)$ w.r.t. $t$.
Recall the DPP given in \eqref{1.08}, for any stopping time $\tau\in[t,T]$, we have
$$u(x,t)=\sup_{\alpha\in\mathcal {A}(x,t)}\mathbf{E}\left[\int^\tau_t\frac{1}{2}|\alpha_r|^2+F(X_\tau^{x,t})dr+u(X_\tau^{x,t},\tau)\right].$$
Fixed $\alpha^*_s$, for any $s\in[t,T]$, we have
$$u(x,t)\geq\mathbf{E}\left[\int^s_t\frac{1}{2}|\alpha^*_r|^2+F(X_r^{x,t})dr+u(X_s^{x,t},s)\right].$$
Then we get
\begin{align*}
 |u(x,s)-u(x,t)|  &\leq\left|\mathbf{E}\left[ u(x,s)-u(X^{x,t}_s,s)\right]\right|+\left|\mathbf{E}\left[u(X^{x,t}_s,s)-u(x,t)\right]\right|\\
  &\leq\left|\mathbf{E}\left[X^{x,t}_s-x\right]\right|+\left|\mathbf{E}\bigg[\int^s_t\frac{1}{2}|\alpha^*_r|^{2}+F(X^{x,t}_r)dr\bigg]\right|.
\end{align*}
Since $\mathbf{E}\big[\int^s_t\sigma_2(X^{x,t}_\tau)dB_{2,\tau}\big]=0$ and \eqref{2.25}, we get
  \begin{align}\label{x1}
    \mathbf{E}\left[X^{x,t}_{2,s}-x_2\right] &
  =\mathbf{E}\bigg[\int^s_th(X^{x,t}_{1,\tau})\alpha^*_{2,\tau} d\tau+\int^s_t\sigma_2(X^{x,t}_\tau)dB_{2,\tau}\bigg]\\\nonumber
    &\leq C\|h\|_{C^2(\mathbb{R})}|\alpha_2^*|_{\infty}|s-t|.
  \end{align}
In the similar calculus \eqref{x1}, we have
$$ \mathbf{E}\left[X^{x,t}_{1,s}-x_1\right] \leq C|\alpha_1^*|_{\infty}|s-t|.$$
Since $h$, $F$ and $\alpha^*$ are bounded,  we have
\begin{align*}
 |u(x,s)-u(x,t)| &\leq C|s-t|\left[|\alpha^*|_{\infty}^{2}+\|F\|_{C^2(\mathbb{R}^2)}+\|h\|_{C^2(\mathbb{R})}|\alpha^*|_{\infty}\right]\leq C|s-t|,
\end{align*}
 and the result follows.
\end{proof}
\subsection{Proof of Lemma \ref{2.3}}
\begin{proof}[{\bf{Proof of Lemma \ref{2.3}}}]
For any $x,y\in\mathbb{R}^{2}$ and $\lambda\in[0,1]$. Consider that
$$x^\lambda=(x^{\lambda}_1,x^\lambda_{2}):=\lambda x+(1-\lambda)y=(\lambda x_1+(1-\lambda)y_1,\lambda x_2+(1-\lambda)y_2),$$
then we have
\begin{equation}\label{8.02}
  \lambda (x-x^\lambda)+(1-\lambda)(y-x^\lambda)=0.
\end{equation}
Let $\alpha\in\mathscr{A}(x,t)$ be an $\varepsilon$-optimal control for $u(x^\lambda,t)$, we set
$$X^{x^\lambda,t}_s=(X^{x^\lambda,t}_{1,s},X^{x^\lambda,t}_{2,s}):=x^{\lambda}+\int^s_tb(X^{x^\lambda,t}_{1,\tau},\alpha_\tau)d\tau+\int^s_t\sigma(X^{x^\lambda,t}_\tau)dB_\tau.$$
Let $X^{x,t}_s$ and $X^{y,t}_s$ satisfy the SDE \eqref{eq1.7} with the $\varepsilon$-optimal control $\alpha$ for $u(x^\lambda,t)$.
We have to estimate $\lambda u(x,t)+(1-\lambda)u(y,t)-u(x^\lambda,t)$. Since the value function satisfies \eqref{2.01}, we just  need to prove
\begin{equation}\label{F}
  \mathbf{E}\big[\lambda F(X^{x,t}_s,m_s)+(1-\lambda) F(X^{y,t}_s,m_s)-F(X^{x^\lambda,t}_{s},m_s)\big]\leq C\lambda(1-\lambda)|x-y|^2,
\end{equation}
 and
 \begin{equation}\label{G}
  \mathbf{E}\big[\lambda G(X^{x,t}_T,m_T)+(1-\lambda) G(X^{y,t}_T,m_T)-G(X^{x^\lambda,t}_T,m_T)\big]\leq C\lambda(1-\lambda)|x-y|^2.
\end{equation}

In the following, we provide the explicit calculations for the second component $X_{2,s}$, and then we could obtain the analogous ones for $X_{1,s}$.

 First, we have
\begin{align}\label{2.23}
 X_{2,s}^{x,t}  &=x_2+\int^s_th(X^{x,t}_{1,\tau})\alpha_{2,\tau}d\tau+\int^s_t\sigma_2(X^{x,t}_\tau)dB_{2,\tau}  \\ \nonumber
               &=X_{2,s}^{x^\lambda,t}+x_2-x^\lambda_2+\int^s_t\left(h(X^{x,t}_{1,\tau})-h(X^{x^\lambda,t}_{1,\tau})\right)\alpha_{2,\tau}d\tau\\ \nonumber
               &\quad+\int^s_t\sigma_2(X^{x,t}_\tau)-\sigma_2(X^{x^\lambda,t}_\tau)dB_{2,\tau},
\end{align}
and analogously for $X^{y,t}_{2,s}$, $X^{x,t}_{1,s}$, and $X^{y,t}_{1,s}$.

To prove the \eqref{F}, taking the Taylor expansion of $F$ centered in $X^{x^\lambda,t}_{s}$, we have
\begin{align}\label{E}
  &\mathbf{E}\left[\lambda F(X^{x,t}_{s})+(1-\lambda) F(X^{y,t}_{s})- F(X^{x^\lambda,t}_{s})\right] \\ \nonumber
   =&\mathbf{E}\big[\lambda\big(F(X^{x^{\lambda},t}_{s})+\partial_{x_1}F(X^{x^{\lambda},t}_{s})(X_{1,s}^{x,t}-X_{1,s}^{x^\lambda,t})
   +\partial_{x_2}F(X^{x^{\lambda},t}_{s})(X_{2,s}^{x,t}-X_{2,s}^{x^\lambda,t})+R_{1}\big) \\ \nonumber
   &+(1-\lambda)\big(F(X^{x^{\lambda},t}_{s})+\partial_{x_1}F(X^{x^{\lambda},t}_{s})(X_{1,s}^{y,t}-X_{1,s}^{x^\lambda,t})
   +\partial_{x_2}F(X^{x^{\lambda},t}_{s})(X_{2,s}^{y,t}-X_{2,s}^{x^\lambda,t})+R_{2}\big)\\ \nonumber
   &-F(X^{x^\lambda,t}_{s})\big]\\ \nonumber
   =&\mathbf{E}\left[\partial_{x_1}F(X^{x^\lambda,t}_{s})\big(\lambda (X_{1,s}^{x,t}-X_{1,s}^{x^\lambda,t})+(1-\lambda)(X_{1,s}^{y,t}-X_{1,s}^{x^\lambda,t})\big)\right]\\ \nonumber
   &+\mathbf{E}\left[\partial_{x_2} F(X^{x^\lambda,t}_{2,s})\big(\lambda (X_{2,s}^{x,t}-X_{2,s}^{x^\lambda,t})+(1-\lambda)(X_{2,s}^{y,t}-X_{2,s}^{x^\lambda,t})\big)\right]+\mathbf{E}\left[\lambda R_{1}+(1-\lambda)R_{2}\right]\\ \nonumber
   :=&I_{x_1}+I_{x_2}+I_{R},
\end{align}
where $R_{1}$ and $R_{2}$ are the error terms of the expansion given by
\begin{align*}
R_{1}=&\frac{1}{2}\partial^2_{x_1}F(\xi_1)(X_{1,s}^{x,t}-X_{1,s}^{x^\lambda,t})^2+\frac{1}{2}\partial^2_{x_2}F(\xi_1)(X_{2,s}^{x,t}-X_{2,s}^{x^\lambda,t})^2\\
&+\partial^2_{x_1x_2}F(\xi_1)(X_{1,s}^{x,t}-X_{1,s}^{x^\lambda,t})(X_{2,s}^{x,t}-X_{2,s}^{x^\lambda,t})
\end{align*}
and
\begin{align*}
R_{2}=&\frac{1}{2}\partial^2_{x_1}F(\xi_2)(X_{1,s}^{y,t}-X_{1,s}^{x^\lambda,t})^2+\frac{1}{2}\partial^2_{x_2}F(\xi_2)(X_{2,s}^{y,t}-X_{2,s}^{x^\lambda,t})^2\\
&+\partial^2_{x_1x_2}F(\xi_2)(X_{1,s}^{y,t}-X_{1,s}^{x^\lambda,t})(X_{2,s}^{y,t}-X_{2,s}^{x^\lambda,t})
\end{align*}
for suitable $\xi_1,\xi_2\in \mathbb{R}^{2}$.
For the term $I_{x_2}$, by \eqref{2.23}, we get
\begin{align*}
I_{x_2}=\mathbf{E}\left[\partial_{x_2}F(X^{x^\lambda,t}_{2,s})\bigg(\int^s_t\alpha_{2,\tau}I_{1}(\tau)d\tau+\int^s_tI_{2}(\tau)dB_{2,\tau}\bigg)\right],
\end{align*}
where
$$I_{1}(\tau)=\lambda\big(h(X_{1,\tau}^{x,t})-h(X_{1,\tau}^{x^\lambda,t})\big)+(1-\lambda)\big(h(X_{1,\tau}^{y,t})-h(X_{1,\tau}^{x^\lambda,t})\big),$$
$$I_{2}(\tau)=\lambda\big(\sigma_2(X^{x,t}_\tau)-\sigma_2(X^{x^\lambda,t}_\tau)\big)+(1-\lambda)\big(\sigma_2(X_{\tau}^{y,t})
-\sigma_2(X_{\tau}^{x^\lambda,t})\big).$$
Now our task is to estimate on $I_{1}$, $I_{2}$.

For the estimates $I_{1}$, the Taylor expansion for $h$ centered in $X^{x^\lambda,t}_{1,\tau}$ yields, \begin{align*}
  I_{1}(\tau) & =h'(X^{x^\lambda,t}_{1,\tau})\left(\lambda\big(X^{x,t}_{1,\tau}-X^{x^\lambda,t}_{1,\tau}\big)+(1-\lambda)\big(X^{y,t}_{1,\tau}
  -X^{x^\lambda,t}_{1,\tau}\big)\right)\\
 &\quad+\lambda h''(\eta_1)\big(X^{x,t}_{1,\tau}-X^{x^\lambda,t}_{1,\tau}\big)^2+(1-\lambda)h''(\eta_2)\big(X^{y,t}_{1,\tau}-X^{x^\lambda,t}_{1,\tau}\big)^2,
\end{align*}
for suitable $\eta_1,\eta_2\in\mathbb{R}$. Taking the expectation for $|I_1|^2$, we have
\begin{align}
 \mathbf{E}\big[| I_{1}(\tau)|^2\big]&\leq C\|h\|^2_{C^2(\mathbb{R})}\mathbf{E}\bigg[\sup_{\tau\in[t,s]}\bigg\{\left(\lambda(X^{x,t}_{1,\tau}-X^{x^\lambda,t}_{1,\tau})
  +(1-\lambda)(X^{y,t}_{1,\tau}-X^{x^\lambda,t}_{1,\tau})\right)^2\\\nonumber
  &\quad+\lambda^2\left(X_{1,\tau}^{x,t}-X^{x^\lambda,t}_{1,\tau}\right)^4
  +(1-\lambda)^2\left(X_{1,\tau}^{y,t}-X^{x^\lambda,t}_{1,\tau}\right)^4\bigg\}\bigg].
\end{align}
It follows from the SDE \eqref{eq1.7}, Lemma \ref{SDE} and \eqref{8.02} that
\begin{align}\label{8.01}
   & \quad\mathbf{E}\bigg[\sup_{\tau\in[t,s]}\bigg\{\left(\lambda(X^{x,t}_{1,\tau}-X^{x^\lambda,t}_{1,\tau})+(1-\lambda)(X^{y,t}_{1,\tau}-X^{x^\lambda,t}_{1,\tau})\right)^2 \bigg\}\bigg] \\\nonumber
 &\leq C\left|\lambda(x_1-x^\lambda_{1})+(1-\lambda)(y_1-x^{\lambda}_1)\right|^2=0.
\end{align}
Since $x_2-x^\lambda_2=(1-\lambda)(x_2-y_2)$, and $y_2-x^\lambda_2=\lambda(y_2-x_2)$, Corollary \ref{SD} gives that
\begin{align}\label{8.03}
   &\quad\mathbf{E}\bigg[\sup_{\tau\in[t,s]}
\bigg\{\lambda^2\left(X_{1,\tau}^{x,t}-X^{x^\lambda,t}_{1,\tau}\right)^4
  +(1-\lambda)^2\left(X_{1,\tau}^{y,t}-X^{x^\lambda,t}_{1,\tau}\right)^4\bigg\}\bigg] \\\nonumber
  &\leq C\lambda^2(x_1-x_1^\lambda)^4+C(1-\lambda)^2(y_1-x_1^\lambda)^4\\\nonumber
  &\leq C\lambda^2(1-\lambda)^2(x_1-y_1)^4.
\end{align}
Combined with \eqref{8.01} and \eqref{8.03}, we get
\begin{equation}\label{2.26}
\big(\mathbf{E}\left[|I_{1}(\tau)|^2\right]\big)^{\frac{1}{2}}\leq C\lambda(1-\lambda)(x_1-y_1)^2.
\end{equation}
According to Tonelli Theorem and the H\"{o}lder inequality, the assumption {\bf{(H5)}} and \eqref{2.26} give that
\begin{align}\label{28}
  \mathbf{E}\bigg[\partial_{x_2}F(X^{x^\lambda,t}_{2,s})\int^s_t\alpha_{2,\tau}I_{1}(\tau)d\tau\bigg]& \leq 
  \|F\|_{C^2(\mathbb{R}^2)}\int^s_t\mathbf{E}\big[|\alpha_{2,\tau}I_{1}(\tau)|\big]d\tau\\\nonumber
 & \leq C\int^s_t \Big(\mathbf{E}\big[|\alpha_{2,\tau}|^2\big]\mathbf{E}\big[ |I_{1}(\tau)|^2\big]\Big)^{\frac{1}{2}}d\tau\\\nonumber
 &\leq C\lambda(1-\lambda)(x_1-y_1)^2 \int^s_t\Big(\mathbf{E}\big[|\alpha_{2,\tau}|^2\big]\Big)^{\frac{1}{2}}d\tau.
\end{align}
Since $\alpha\in\mathcal {A}(x,t)$ satisfies \eqref{2.28}, by Tonelli Theorem and the H\"{o}lder inequality, we have
\begin{align}\label{6.1}
  \int^s_t\Big(\mathbf{E}\big[|\alpha_{2,\tau}|^2\big]\Big)^{\frac{1}{2}}d\tau &\leq |s-t|^{\frac{1}{2}}\left(\mathbf{E}\left[\int^s_t|\alpha_{2,\tau}|^2d\tau\right]\right)^{\frac{1}{2}}\leq C,
\end{align}
Hence, combined with \eqref{28} and \eqref{6.1}, we have
\begin{equation}\label{6.12}
   \mathbf{E}\bigg[\partial_{x_2}F(X^{x^\lambda,t}_{2,s})\int^s_t\alpha_{2,\tau}I_{1}(\tau)d\tau\bigg]\leq C\lambda(1-\lambda)(x_1-y_1)^2.
\end{equation}

The estimate $I_{2}$ refers to $I_1$. Similar to \eqref{6.12}, by Lemma \ref{Norm}, we have
\begin{align}\label{3.37}
  \mathbf{E}\bigg[\partial_{x_2}F(X^{x^\lambda,t}_{2,s})\int^s_tI_{2}(\tau)dB_{2,\tau}\bigg]
  &\leq C\|F\|_{C^2(\mathbb{R}^2)}\left(\mathbf{E}\bigg[\Big|\int^s_t I_{2}(\tau)dB_{2,\tau}\Big|^2\bigg]\right)^{\frac{1}{2}}\\\nonumber
 &=C\left(\int^s_t\mathbf{E} \left[|I_{2}(\tau)\big|^2\right]d\tau\right)^{\frac{1}{2}}.
\end{align}
 Since $\|\sigma_2\|_{C^2(\mathbb{R}^2)}\leq C$ provided by the assumption {\bf{(H3)}}, we take the Taylor expansion for $\sigma_2(X_s)$. Then by the same argument with \eqref{2.26},  we get
\begin{equation}\label{3.34}
 \mathbf{E}\left[\big|I_{2}(\tau)\big|^2\right]\leq C\lambda^2(1-\lambda)^2|x-y|^4.
\end{equation}
It follows from \eqref{3.37} and \eqref{3.34} that
\begin{align}\label{2.9}
  \mathbf{E}\bigg[\partial_{x_2}F(X^{x^\lambda,t}_{2,s})\int^s_tI_{2}(\tau)dB_{2,\tau}\bigg]&\leq C\bigg(\int^s_t\mathbf{E} \big[|I_{2}(\tau)|^2\big]d\tau\bigg)^{\frac{1}{2}}\\\nonumber
  &\leq C|s-t|^{\frac{1}{2}}\lambda(1-\lambda)|x-y|^2.
\end{align}
Combining \eqref{6.12} and \eqref{2.9}, we have
\begin{align}\label{2.11}
I_{x_2}\leq C \lambda(1-\lambda)|x-y|^2.
\end{align}

By the similar discussion as above and $\|\sigma_1\|_{C^2(\mathbb{R}^2)}\leq C$ provided by the assumption {\bf{(H3)}}, we have
\begin{align}\label{2.12}
I_{x_1}\leq C \lambda(1-\lambda)|x-y|^2.
\end{align}

Let us estimate the error term $I_{R}$. We have
\begin{align*}
I_{R}=&\mathbf{E}[\lambda R_{1}+(1-\lambda)R_{2}]\\
=&\frac{1}{2}\mathbf{E}\left[\lambda\partial^2_{x_2}F(\xi_1)\big(X_{1,s}^{x,t}-X_{1,s}^{x^\lambda,t}\big)^2 +(1-\lambda)\partial^2_{x_2}F(\xi_2)\big(X_{1,s}^{y,t}-X_{1,s}^{x^\lambda,t}\big)^2\right]\\
&+\frac{1}{2}\mathbf{E}\left[\lambda\partial^2_{x_2}F(\xi_1)\big(X_{2,s}^{x,t}-X_{2,s}^{x^\lambda,t}\big)^2 +(1-\lambda)\partial^2_{x_2}F(\xi_2)\big(X_{2,s}^{y,t}-X_{2,s}^{x^\lambda,t}\big)^2\right]\\
&+\mathbf{E}\big[\lambda\partial^2_{x_1x_2}F(\xi_1)(X_{1,s}^{x,t}-X_{1,s}^{x^\lambda,t})(X_{2,s}^{x,t}-X_{2,s}^{x^\lambda,t}) \\ &+(1-\lambda)\partial^2_{x_1x_2}F(\xi_2)(X_{1,s}^{y,t}-X_{1,s}^{x^\lambda,t})(X_{2,s}^{y,t}-X_{2,s}^{x^\lambda,t})\big]\\
=&I_{R_{1}}+I_{R_{2}}+I_{R_{3}}.
\end{align*}
Similar to \eqref{8.03}, Corollary \ref{SD} gives
\begin{align}\label{2.101}
I_{R_2}&\leq \frac{1}{2}\|F\|_{C^2(\mathbb{R}^2)}\mathbf{E}\left[\sup_{t\leq s\leq T}\big\{\lambda(X_{2,s}^{x,t}-X_{2,s}^{x^\lambda,t})^2 +(1-\lambda)(X_{2,s}^{y,t}-X_{2,s}^{x^\lambda,t})^2\big\}\right]\\\nonumber
  &\leq C\lambda(x_2-x^\lambda_2)^2+C(1-\lambda)(y_2-x^\lambda_2)^2\\\nonumber
  &\leq C\lambda(1-\lambda)(x_2-y_2)^2.
\end{align}
Similarly, we get $I_{R_1}\leq C\lambda(1-\lambda)(x_1-y_1)^2$. For the term $I_{R_3}$, using the inequality $a^{2}+b^{2}\geq 2ab$, we have
\begin{align*}
I_{R_3}\leq & C\mathbf{E}\left[\sup_{t\leq s\leq T}\big\{\lambda(X_{1,s}^{x,t}-X_{1,s}^{x^\lambda,t})(X_{2,s}^{x,t}-X_{2,s}^{x^\lambda,t}) +(1-\lambda)(X_{1,s}^{y,t}-X_{1,s}^{x^\lambda,t})(X_{2,s}^{y,t}-X_{2,s}^{x^\lambda,t})\big\}\right]\\\nonumber
\leq & C\mathbf{E}\left[\sup_{t\leq s\leq T}\big\{\lambda(X_{1,s}^{x,t}-X_{1,s}^{x^\lambda,t})^2 +(1-\lambda)(X_{1,s}^{y,t}-X_{1,s}^{x^\lambda,t})^2\big\}\right]\\\nonumber
  &+ C\mathbf{E}\left[\sup_{t\leq s\leq T}\big\{\lambda(X_{2,s}^{x,t}-X_{2,s}^{x^\lambda,t})^2 +(1-\lambda)(X_{2,s}^{y,t}-X_{2,s}^{x^\lambda,t})^2\big\}\right]\\\nonumber
\leq & C\lambda(1-\lambda)|x-y|^2.
\end{align*}
Collecting the above estimates for $I_{R_1}$, $I_{R_2}$ and $I_{R_3}$, we have
\begin{align}\label{2.13}
I_{R}\leq C\lambda(1-\lambda)|x-y|^2.
\end{align}

Taking into account \eqref{E}, \eqref{2.11}, \eqref{2.12}, \eqref{2.13} and $\|F\|_{C^2(\mathbb{R}^2)}\leq C$, we get \eqref{F}. By the same argument for the function $G$, we get \eqref{G}. Hence the semiconcavity of $u$ holds, and the result follows.
\end{proof}
\section{THE VISCOSITY VANISHING LIMIT}
In this section, we prove the existence and uniqueness of the vanishing viscosity limit for the MFG systems \eqref{1.2} given in Proposition \ref{3.1}. To begin with this progress, we give Lemma \ref{Lem3.1}-Lemma \ref{Lem3.4} for the auxiliary systems \eqref{4.2}, which will be used later.

For the auxiliary MFG systems \eqref{4.2}, there exists a unique coupling solutions $(u^\epsilon, m^\epsilon)$. Specifically, fixed the measure $\bar{m}\in C([0,T];\mathcal {P}_1)$, the HJE is given by
\begin{equation}\label{5.1}
\left\{
  \begin{array}{ll}
     -\partial_t u-\big(\epsilon\Delta+\mathcal {L}\big)u+\frac{1}{2}|D_Gu|^2=F(x,\bar{m}),\ \quad\ \ \text{in}\ \mathbb{R}^{2}\times (0,T), \\
    u(x,T)=G(x,\bar{m}_T),\qquad\qquad\qquad\qquad\qquad\quad\ x\in\mathbb{R}^{2}.
  \end{array}
\right.
\end{equation}
The corresponding FPE is given by
\begin{equation}\label{5.3}
\left\{
  \begin{array}{ll}
    \partial_t m-\big(\epsilon\Delta+\mathcal {L}^*\big)m-\mathrm{div}_G(mD_Gu)=0,\qquad\quad\text{in}\ \mathbb{R}^{2}\times (0,T), \\
     m(x,0)=m_0(x),\qquad\qquad\qquad\qquad\qquad\qquad\ \ \ x\in\mathbb{R}^{2}.
  \end{array}
\right.
\end{equation}

Our aim is to find a solution to the original systems \eqref{1.2} letting $\epsilon\rightarrow0^+$. As a first step, we establish the well-posedness of the auxiliary systems \eqref{4.2}.
\begin{Lem}\label{Lem3.1}
Under assumptions {\bf{(H1)}}-{\bf{(H5)}}, for any $\bar{m}\in C([0,T],\mathcal {P}_1)$, there exists a unique bounded classical solution $(u^{\epsilon},m^{\epsilon})$ to the auxiliary systems \eqref{4.2}. Moreover, $m^\epsilon>0$.
\end{Lem}
\begin{proof}
The proof follows Lemma 3.1 in \cite{0} by using standard regularity results for quasi-linear parabolic equations. First we claim that the solution $u^\epsilon$ of the equation \eqref{5.1} is bounded in $\mathbb{R}^{2}\times[0,T]$, that is
\begin{equation}\label{3}
\|u^\epsilon\|_\infty\leq C.
\end{equation}
In fact, if $C$ is sufficiently large, the functions $w^{\pm}(t):=C\pm C(T-t)$ are respectively a supersolution and a subsolution for equation \eqref{5.1}. Then the claim \eqref{3} easily follows from comparison principle and assumptions {\bf{(H2)}}, {\bf{(H3)}}, {\bf{(H5)}}.
 Hence we apply Theorem 8.1 in \cite{47} to obtain the existence and uniqueness of a classical solution $u^\epsilon$ in $\mathbb{R}^{2}\times[0,T]$.

 Now $m^\epsilon$ is the classical solution of the linear equation
$$\partial_t m-\big(\epsilon\Delta+\mathcal {L}^*\big)m- D_Gm D_Gu -m\Delta_G u=0,\ m(x,0)=m_0(x),$$
with H\"{o}lder continuous coefficients. Hence we apply Theorem 5.1 in \cite{47} to obtain the existence and uniqueness of a classical solution $m^\epsilon$ of \eqref{5.3}. From assumptions on $m_0$ and the maximum principle given by Theorem 2.1 in \cite{47}, we have $m^\epsilon>0$.
\end{proof}
Set a direction vector
\begin{align}\label{5.31}
\eta=(\eta_{1},\eta_{2}) \ \text{with} \ |\eta|=1.
\end{align}
For the function $f:\mathbb{R}^{2}\rightarrow\mathbb{R}$, we use the shorthand
$$f_{\eta}:=Df\cdot\eta,\ \text{and}\ f_{\eta\eta}:=D^{2}f\eta\cdot\eta.$$
In the following, we prove some useful properties for the auxiliary systems \eqref{4.2}.
\begin{Lem}\label{Lem3.2}
Under the same assumptions of Lemma \ref{Lem3.1}, for any direction vector $\eta\in \mathbb{R}^{2}$ defined in \eqref{5.31}, there exists a constant $C>0$ independent of $\epsilon$ such that
$$\|u^\epsilon_{\eta}\|_\infty\leq C\quad \text{and}\quad  \|u^\epsilon_{\eta\eta}\|_\infty\leq C.$$
\end{Lem}
\begin{proof}
First we prove the uniform Lipschitz continuity of $u^\epsilon$ as Lemma \ref{Lem2.2}. Because of the adding of the term $\epsilon\Delta u^\epsilon$, we just need to modify the diffusion term as follows
\begin{align}\label{eq7.1}
\left\{
 \begin{array}{ll}
  dZ_{1,s}=\alpha_{1,s}ds+\sigma^\epsilon_1(Z_s)dB_{1,s},\\
  dZ_{2,s}=\alpha_{2,s}h(X_{1,s})ds+\sigma^\epsilon_2(Z_s)dB_{2,s},\\
  Z_{1,t}=x_1,\ Z_{2,t}=x_2,
\end{array}
\right.
\end{align}
where $\sigma^\epsilon_i:=\sqrt{2\epsilon+\sigma_i^2}$, $i=1,2$. Moreover, $\sigma_i^\epsilon$ is also Lipschitz continuous and satisfies $\|\sigma^\epsilon\|_{C^2(\mathbb{R}^2)}\leq C$ by the Lipschitz continuity of $\sigma_i$ and {\bf{(H3)}}.
Then we have $|Du^{\epsilon}|\leq C$, hence we get $\|u^\epsilon_{\eta}\|_\infty\leq C$.

 Then set the matrix $$A(x):=\frac{1}{2}\sigma(x)\sigma'(x)$$ and $A_{\eta}(x)$ represents the directional derivative of the entries of the matrix, similarly denote $A_{\eta\eta}(x)$. 
 Compute the derivative of the equation \eqref{5.1} twice w.r.t. $\eta$, referring the idea of Lemma 2.2 in \cite{Com}. By the assumptions {\bf{(H2)}} and {\bf{(H5)}},
 since $|D_Gu^{\epsilon}|^2_\eta=2D_Gu^{\epsilon}D_Gu^{\epsilon}_\eta+(h^{2})_{\eta}(\partial_{x_{2}}u^{\epsilon})^{2}$, we have
\begin{align}\label{1}
&\quad-\partial_{t}u^{\epsilon}_{\eta}-\epsilon\Delta u^{\epsilon}_{\eta}-\mathbf{tr}(AD^{2}u^{\epsilon}_{\eta}+A_{\eta}D^{2}u^{\epsilon})+D_{G}u^{\epsilon}  D_{G}u^{\epsilon}_{\eta}\\\nonumber
&=F_{\eta}(x,\bar{m})-\frac{1}{2}(h^{2})_{\eta}(\partial_{x_{2}}u^{\epsilon})^{2}\leq C,
\end{align}
and
\begin{align}\label{3.2}
&-\partial_{t}u^{\epsilon}_{\eta\eta}-\epsilon\Delta u^{\epsilon}_{\eta\eta}-\mathbf{tr}(AD^{2}u^{\epsilon}_{\eta\eta}+2A _{\eta}D^{2}u^{\epsilon}_{\eta}+A_{\eta\eta}D^{2}u^{\epsilon})+D_{G}u^{\epsilon}  D_{G}u^{\epsilon}_{\eta\eta}\\\nonumber
=&F_{\eta\eta}(x,\bar{m})-|D_{G}u^{\epsilon}_{\eta}|^{2}-\frac{1}{2}(h^{2})_{\eta\eta}(\partial_{x_{2}} u^{\epsilon})^{2}
-2(h^{2})_{\eta}\partial_{x_{2}}u^{\epsilon}\partial_{x_{2}} u^{\epsilon}_{\eta}\\\nonumber
\leq&F_{\eta\eta}(x,\bar{m})-|D_{G}u^{\epsilon}_{\eta}|^{2}-C\left(1+|D_{G}u^{\epsilon}_{\eta}|\right).
\end{align}
Since $-|D_{G}u^{\epsilon}_{\eta}|^{2}-C(1+|D_{G}u^{\epsilon}_{\eta}|)$ is bounded above by a constant, we deduce
\begin{align*}
&-\partial_{t}u^{\epsilon}_{\eta\eta}-\epsilon\Delta u^{\epsilon}_{\eta\eta}-\mathbf{tr}(AD^{2}u^{\epsilon}_{\eta\eta}+2A_{\eta}D^{2}u^{\epsilon}_{\eta}+A_{\eta\eta}D^{2}u^{\epsilon})+D_{G}u^{\epsilon}  D_{G}u^{\epsilon}_{\eta\eta}\leq C.
\end{align*}

Now we construct the auxiliary function
\begin{equation}\label{B}
  \hat{u}(x,t,\eta):=u^{\epsilon}_{\eta}+u^{\epsilon}_{\eta\eta}.
\end{equation}
Then by the calculations, we have
\begin{align*}
D^{2}_{(x,\eta)}\hat{u}=\left(
                  \begin{array}{cc}
                    D^{2}u^{\epsilon}_{\eta}+D^{2}u^{\epsilon}_{\eta\eta} & D^{2}u^{\epsilon}+2D^{2}u^{\epsilon}_{\eta} \\
                    D^{2}u^{\epsilon}+2D^{2}u^{\epsilon}_{\eta} & 2D^{2}u^{\epsilon} \\
                  \end{array}
                \right).
\end{align*}
Set
\begin{align}\label{m1}
\hat{A}=\hat{A}(x,t,\eta):=\frac{1}{2}\left(
                  \begin{array}{cc}
                    2A &A_{\eta} \\
                    A_{\eta} & A_{\eta\eta} \\
                  \end{array}
                \right),
\end{align}
then it follows from \eqref{1} and \eqref{3.2} that
\begin{align*}
-\partial_{t}\hat{u}-\epsilon\Delta \hat{u}-\mathbf{tr}(\hat{A}D^{2}\hat{u})+D_{G}u^{\epsilon}  D_{G}\hat{u}\leq C.
\end{align*}
Since $\|\hat{u}(\cdot,T)\|_{\infty}\leq C$ by the assumption {\bf{(H5)}}, then we can conclude by comparison that $\|\hat{u}\|_{\infty}\leq C$ for a constant $C$ is independent of $\epsilon$. Further, since $u^{\epsilon}_{\eta\eta}=\hat{u}-u^{\epsilon}_{\eta}$ and $\|u^{\epsilon}_{\eta}\|_{\infty}\leq C$, we obtain $\|u^{\epsilon}_{\eta\eta}\|_{\infty}\leq C$, which completes the proof.
\end{proof}
\begin{Rem}\label{weiR}
Lemma \ref{Lem3.2} implies that $\partial_{x_i} u^\epsilon\leq C$ and $\partial_{x_i}^2 u^\epsilon\leq C$. By the boundness of $h$ given in {\bf{(H2)}}, we have
\begin{equation}\label{wei1}
\Delta_G u^\epsilon\leq C, \,\, |D_G u^\epsilon|\leq C, \,\, \textup{and}\,\, |D^2_G u^\epsilon|\leq C.
\end{equation}
\end{Rem}
\begin{Rem}
By Proposition 1.1.3 (e) in \cite{Semi} and $\|u^\epsilon_{\eta\eta}\|_\infty\leq C$ given in Lemma \ref{Lem3.2}, we conclude that $u^\epsilon$ is uniformly semiconcave in $\mathbb{R}^{2}$, which is a different proof from the method of Lemma \ref{2.3}.
\end{Rem}
\begin{Lem}\label{Lem3.3}
Under assumptions {\bf{(H2)-(H5)}}, there exists a constant $C>0$ independent of $\epsilon$ such that
\begin{enumerate}
  \item[\rm{(i)}] $\|m^\epsilon\|_{\infty}\leq C,$
  \item[\rm{(ii)}] $d_{1}(m^{\epsilon}_{t_1},m^{\epsilon}_{t_2})\leq C|t_1-t_2|^{\frac{1}{2}}$, for any $t_1, t_2\in(0,T),$
  \item[\rm{(iii)}] $\int_{\mathbb{R}^{2}}|x|^2dm^{\epsilon}_t\leq C(\int_{\mathbb{R}^{2}}|x|^2dm^{\epsilon}_{0}+1)$, for any $t\in(0,T).$
\end{enumerate}
\end{Lem}
\begin{proof}
First, we prove (i) referring to Lemma 3.1 in \cite{0}.
 By Remark \ref{weiR} and $m\geq0$, we have
\begin{align*}
  \mathrm{div}_G(m^\epsilon D_Gu^\epsilon)= D_G u^\epsilon  D_G m^\epsilon+m^\epsilon\Delta_G u^\epsilon\leq D_G u^\epsilon  D_G m^\epsilon+Cm^\epsilon.
\end{align*}
Therefore, by assumptions {\bf{(H2)}} and {\bf{(H3)}}, the function $m$ satisfies
\begin{align*}
 \partial_t m^\epsilon-\epsilon\Delta m^\epsilon-\sum^2_{i=1}\left(\frac{1}{2}\sigma^2_{i}\partial^2_{x_i}m^\epsilon+\partial_{x_i} \sigma^2_{i}\partial_{x_i} m^\epsilon\right)&= \frac{1}{2}\sum^2_{i=1}m^\epsilon\partial^2_{x_i}\sigma^2_{i}+\mathrm{div}_G(m^\epsilon D_Gu^\epsilon)\\
  &\leq D_G u^\epsilon {D_G} m^\epsilon+Cm^\epsilon,
\end{align*}
with $m^\epsilon(x,0)\leq C$.
Using $w=Ce^{Ct}$ as supersolution, where $C$ is independent of $\epsilon$, we infer $\|m^\epsilon\|_{\infty}\leq w=Ce^{Ct}$,
which provided by comparison principle for subsolution $m^\epsilon$.

Second, we prove (ii) referring to Lemma 3.4 in \cite{Cardaliaguet}. Denote the distribution of $Z_t$ defined in SDE \eqref{eq7.1} by $m^\epsilon_t$. By standard arguments of Lemma 3.3 in \cite{Cardaliaguet}, $m^\epsilon_t$ is a weak solution to the equation \eqref{5.3}. By the definition of $d_1$, we note that the law $\gamma$ of the pair $(Z_t,Z_s)$ belongs to $\Pi(m_t,m_s)$, so that
$$ d_1(m^{\epsilon}_t,m^{\epsilon}_s) \leq \int_{\mathbb{R}^{2}\times\mathbb{R}^{2}}|x-y|d\gamma(x,y)=\mathbf{E}\big[|Z_t-Z_s|\big].$$

For instance $t<s$, since the optimal control $\alpha^*=D_Gu^\epsilon$ is bounded by \eqref{2.25}, it follows from \eqref{eq7.1}, Jensen inequality and Lemma \ref{Norm} that
\begin{align*}
\mathbf{E}\big[|Z_{2,s}-Z_{2,t}|\big]
&\leq C\|h\|_{C^2(\mathbb{R})}\mathbf{E}\left[\int_t^s|D_Gu^\epsilon| d\tau\right]
+\bigg(\mathbf{E}\bigg[\bigg(\int^s_t|\sigma_2^\epsilon(Z_{\tau})|dB_{2,\tau}\bigg)^2\bigg]\bigg)^{\frac{1}{2}}\\
 &\leq C\|h\|_{C^2(\mathbb{R})}|s-t|
  +\mathbf{E}\left[\int^s_t|\sigma_2^\epsilon(Z_\tau)|^2d\tau\right]^{\frac{1}{2}}\\
  &\leq C\|h\|_{C^2(\mathbb{R})}T^{\frac{1}{2}}|s-t|^{\frac{1}{2}}+\|\sigma^\epsilon\|_{C^2(\mathbb{R}^2)}|s-t|^{\frac{1}{2}}\\
  &\leq C|s-t|^{\frac{1}{2}},
  \end{align*}
and analogously for $|Z_{1,s}-Z_{1,t}|$. Here $C>0$ is independent of $\epsilon$, because of the uniformly boundness of $\sigma^\epsilon$, $h$ and $D_Gu^\epsilon$. Then (ii) holds.

Finally, we prove (iii) by the same argument of Lemma 3.5 in \cite{Cardaliaguet}. We have
\begin{align*}
  \int_{\mathbb{R}^{2}}x_2^2 dm^{\epsilon}_t=\mathbf{E}\left[|Z_{2,t}|^2\right]
  &  \leq C\mathbf{E}\bigg[|Z_{2,t}|^2+\int_{t}^{s}|h(Z_{1,\tau})\alpha_{2,\tau}|^2 d\tau
  +\int_{t}^{s}|\sigma_2^\epsilon(Z_{\tau})|^2d\tau\bigg]\\
  &\leq C\left(\int_{\mathbb{R}^{2}}x_2^2 dm_0+\|h\|_{C^2(\mathbb{R})}^2|s-t|^2+\|\sigma\|_{C^2(\mathbb{R}^2)}^2|s-t|\right)\\
  &\leq C\bigg(\int_{\mathbb{R}^{2}}|x|^2 dm_0+1\bigg),
  \end{align*}
and analogously for $\int_{\mathbb{R}^{2}}x_1^2 dm^{\epsilon}_t$. Then the result follows.
\end{proof}

\begin{Lem}\label{Lem3.4}
Under the same assumptions of Lemma \ref{Lem3.3}, the function $u^{\epsilon}$ is uniformly continuous w.r.t. $t$ and uniformly in $\epsilon$.
\end{Lem}
\begin{proof}
We shall follow the arguments of Lemma 5.1 in \cite{Cardaliaguet13}. Set $u^{\epsilon}_{T}:=u^{\epsilon}(x,T)$, then $u^{\epsilon}_{T}$ are bounded in $C^2(\mathbb{R}^2)$ uniformly in $\epsilon$ by the assumption {\bf{(H2)}}. As Lemma \ref{Lem3.2}, there exists constant $C_0$ such that $w^\pm(x,t):=u^{\epsilon}_{T}(x)\pm C_0(T-t)$ are respectively super- and sub-solution of the equation \eqref{5.1} for any $\epsilon$. Actually, we have
$$-\partial_t w^{+}-\big(\epsilon\Delta+\frac{1}{2}{\bf{tr}}(\sigma^\epsilon \sigma^\epsilon{'}D^2)\big)w^++\frac{1}{2}|D_Gw^+|^2-F(x,\bar{m})\geq C_0-C\epsilon-C\geq0,$$
 and similarly for $w^-$. Hence the comparison principle gives that for any $t\in[0,T]$,
 \begin{equation}\label{6.2}
 \|u^\epsilon(x,t)-u_T^\epsilon(x)\|_{\infty}\leq C_1(T-t).
 \end{equation}

For the equation \eqref{5.1}, the assumption {\bf{(H5)}} and Lemma \ref{Lem3.3} imply that
\begin{align*}
  \sup_{t\in[\tau,T]}\|F(x,\bar{m}_t)-F(x,\bar{m}_{t-\tau})\|_{\infty}
  &\leq C \sup_{t\in[\tau,T]}d_1(\bar{m}_{t},\bar{m}_{t-\tau})=:\delta(\tau),
\end{align*}
and $\delta(\tau)\rightarrow0$ as $\tau\rightarrow0$.
Set $$v^{\epsilon}_{\tau}(x,t):=u^{\epsilon}(x,t-\tau)+C_1\tau+\delta(\tau)(T-t),$$
 then for any $t\in[\tau,T],$ we have
\begin{align*}
   &\quad  -\partial_t v^{\epsilon}_{\tau}(x,t)-\big(\epsilon\Delta+\frac{1}{2}{\bf{tr}}(\sigma^\epsilon\sigma^\epsilon{'}D^2)\big)v^{\epsilon}_{\tau}(x,t)+\frac{1}{2}|{D_G} v^{\epsilon}_{\tau}(x,t)|^2-F(x,\bar{m})(x,t) \\
  &= F(x,\bar{m})(x,t-\tau)-F(x,\bar{m})(x,t)+\delta(\tau)\geq0,
\end{align*}
that means $v^{\epsilon}_{\tau}(x,t)$ is a supersolution of \eqref{5.1}. By the estimate \eqref{6.2}, we have $$v^{\epsilon}_{\tau}(x,T):=u^{\epsilon}(x,T-\tau)+C_1\tau\geq u^{\epsilon}(x,T).$$
Therefore, again by comparison principle, we get $v^{\epsilon}_{\tau}(x,t)\geq u^{\epsilon}(x,t)$, that is $$u^{\epsilon}(x,t)\leq u^{\epsilon}(x,t-\tau)+C_1 \tau+\delta(\tau)(T-t).$$
In a similar way, we also obtain
$$u^{\epsilon}(x,t)\geq u^{\epsilon}(x,t-\tau)-C_1 \tau-\delta(\tau)(T-t),$$
let $\tau\rightarrow0$, accomplishing the proof.
\end{proof}

\begin{proof}[\bf{Proof of Proposition \ref{3.1}}]

As in the proof of Theorem 3.1 in \cite{Cardaliaguet} for the classical systems one can check that the $\{m^\epsilon\}$  all belong to the compact subset $\mathcal {C}$ of $C([0,T],\mathcal {P}_1)$. 
Then \eqref{3} given in Lemma \ref{Lem3.1} implies that $u^\epsilon$ are uniformly bounded.  Then Lemma \ref{2.1}, Lemma \ref{Lem3.2} and Lemma \ref{Lem3.4} give that $u^\epsilon$ locally uniform converges to $u$, $u$ is semiconcave, and $D_Gu^\epsilon$ converges to $D_G u$  a.e.. By standard stability result for viscosity solutions, the function $u$ solves the HJE \eqref{5.2}.

We now need to pass to the limit in the FPE \eqref{5.4}. By the bounds on $m^\epsilon$ given in Lemma \ref{Lem3.3}, as $\epsilon\to0^+$, then $m^\epsilon$ converges to some $m\in \mathcal {C}$ in $C([0,T],\mathcal {P}_1)$ topology and in $L^{\infty}_{loc}$-weak$^*$ topology. Moreover we deduce that $m(x,0)=m_0(x)$. Since $m^\epsilon$ are solutions to \eqref{5.3}, for any $\varphi\in C^\infty_0(\mathbb{R}^2\times(0,T))$, there holds
$$\int^T_0\int_{\mathbb{R}^2}m^\epsilon\big(-\partial_t\varphi-(\epsilon\Delta+\mathcal {L})\varphi +D_Gu^\epsilon   D_G\varphi\big) dxd\tau=0.$$
Letting $\epsilon\to0^+$, by the $L^{\infty}_{loc}$-weak$^*$ convergence of $m^\epsilon$, and by the convergence $D_G u^\epsilon\to D_G u$ a.e., we conclude that the function $m$ solves the equation \eqref{5.4}.


For the uniqueness of the vanishing viscosity limit, one can check it by the comparison principle for the viscosity solution, which is given in Theorem 4.4.5 of \cite{Hu}.
\end{proof}
\section{EXISTENCE AND UNIQUENESS OF THE MFG SYSTEMS}
In this section, we first prove the existence and the uniqueness of the MFG systems \eqref{1.2} given in Theorem \ref{2.4}. Then adding  the assumption \textbf{(H7)}, we obtain the higher regularity given in Theorem \ref{2.44}.

 The uniqueness of the MFG systems \eqref{1.2} holds depending on the monotonicity of $F$ and $G$ in \textbf{(H6)}, which can refer to Theorem 4.3 in \cite{LR}. Now we only prove the existence in Theorem \ref{2.4}, which is based on the Schauder fixed point theorem.
\begin{proof}[\bf{Proof of Theorem \ref{2.4}}]
 For any $\mu\in\mathcal {C}$, we associate $m=\psi(\mu)$ in the following way. By Proposition \ref{3.1}, there exists a unique solution $u$ of the HJE
\begin{equation}\label{3.3}
\left\{
  \begin{array}{ll}
     -\partial_t u-\mathcal {L}u+\frac{1}{2}|D_Gu|^2=F(x,\mu), \qquad\text{in}\ \mathbb{R}^{2}\times (0,T), \\
    u(x,T)=G(x,\mu_T),\qquad\qquad\qquad\qquad x\in\mathbb{R}^{2}.
  \end{array}
\right.
\end{equation}
Define $m:=\psi(\mu)$ as the solution of the FPE
\begin{equation}\label{3.4}
\left\{
  \begin{array}{ll}
    \partial_t m-\mathcal {L}^*m-\mathrm{div}_G(mD_Gu)=0,\qquad\ \ \text{in}\ \mathbb{R}^{2}\times(0,T),\\
     m(x,0)=m_0(x),\qquad\qquad\qquad\qquad\quad\ x\in \mathbb{R}^{2}.
  \end{array}
\right.
\end{equation}
 Then the mapping $\psi$ is single valued by Proposition \ref{3.1}.

 First, let us check that $\psi$ is a well-defined from $\mathcal {C}$ to itself. By the same argument of Lemma \ref{Lem3.3}-(ii), for $X_s$ satisfying \eqref{eq1.7}, $t<s\leq T$, we have $$d_1(m_s,m
  _t)\leq  \mathbf{E}\big[|X_s-X_t|\big]\leq C|s-t|^{\frac{1}{2}}.$$
Thus by the definition \eqref{eq2.7}, $m$ belongs to $\mathcal {C}$, and the mapping $\psi:\mu\rightarrow m=\psi(\mu)$ is well-defined from $\mathcal {C}$ into itself.

Second, let us check that $\psi$ is a continuous map. Let $\mu_n\in \mathcal {C}$ converges to some $\mu$. Let $(u_n,m_n)$ and $(u,m)$ be the corresponding solutions to \eqref{3.3}-\eqref{3.4}, then for any $\varphi\in C^\infty_0(\mathbb{R}^2\times(0,T))$, there holds
\begin{equation}\label{m}
 \int^T_0\int_{\mathbb{R}^2}m_n\big(-\partial_t\varphi-\mathcal {L}\varphi +D_Gu_n D_G\varphi\big) dxd\tau=0.
\end{equation}
By the continuity assumption {\bf{(H5)}} on $F$ and $G$, we get $(x,t)\rightarrow F(x,\mu_n(t))$, $x\rightarrow G(x,\mu_n(T))$ locally uniformly converges to $(x,t)\rightarrow F(x,\mu(t))$, $x\rightarrow G(x,\mu(T))$.
Then one gets the local uniformly convergence of $u_n$ to $u$ by standard arguments of viscosity solutions.
  Since $\{u_n\}_n$ are semiconcave by Lemma \ref{2.3}, which implies that $(u_n)_{\eta\eta}\leq C$ in the distribution sense for any direction vector $\eta\in \mathbb{R}^{2}$ defined in \eqref{5.31}, where $C$ is uniformly for $n$. Then by the local uniform convergence of $u_n$ to $u$, Lemma \ref{2.1} gives that $\{D_Gu_n\}_n$ converges almost everywhere to $D_Gu$ in $\mathbb{R}^2\times(0,T)$. Let $n\rightarrow\infty$ in \eqref{m}, by the $L^\infty_{loc}$-weak* convergence of $m_n$, and by the convergence $D_Gu_n\rightarrow D_Gu$ a.e., then the limit of any converging subsequence of $m_n$ is a weak solution of \eqref{3.4}. But $m$ is the unique weak solution of the equation \eqref{3.4}, which proves that $\{m_n\}$ converges to $m$.

Because $\mathcal {C}$ is compact, the continuous map $\psi$ is compact. We conclude by Schauder fixed point theorem that the compact map  $\mu\rightarrow m=\psi(\mu)$ has a fixed point in $\mathcal {C}$. This fixed point $m$ and its corresponding $u$ is a pair of solutions to the MFG systems \eqref{1.2}. Then the result follows.
\end{proof}

\begin{proof}[{\bf{Proof of Theorem \ref{2.44}}}]
Based on Proposition \ref{3.1} and Theorem \ref{2.4}, we have ${D_G}u^\epsilon\rightarrow {D_G}u$ $a.e.$, then our main task is to prove
\begin{align*}
{D^{2}}u^\epsilon\rightarrow {D^{2}}u \ a.e.,
\end{align*}
where $u^\epsilon$ is the solution of the equation \eqref{5.1}, and $u$ is the solution of the HJE \eqref{5.2}.

\noindent{\bf{Step\ 1.}} Here we claim that under the assumptions {\bf{(H2)-(H5)}} and {\bf{(H7)}}, for any direction vector $\eta\in \mathbb{R}^{2}$ defined in \eqref{5.31}, we have
\begin{equation}\label{D3}
\|u^\epsilon_{\eta\eta\eta}\|_{\infty}\leq C,
\end{equation}
where $C$ is a positive constant independent of $\epsilon$.

In fact, by setting
$$Q(x):=diag\{1, h(x_1)\},$$
we have $D_Gu^\epsilon=Du^{\epsilon} Q$. Compute the derivative of the equation \eqref{5.1} twice w.r.t. $\eta$. Similar to \eqref{3.2}, we have
\begin{align}\label{3.22}
&\quad -\partial_{t}u^{\epsilon}_{\eta\eta}-\epsilon\Delta u^{\epsilon}_{\eta\eta}-\mathbf{tr}(AD^{2}u^{\epsilon}_{\eta\eta}+2A _{\eta}D^{2}u^{\epsilon}_{\eta}+A_{\eta\eta}D^{2}u^{\epsilon})\\\nonumber
&\quad+Du^{\epsilon}Q^2(Du^{\epsilon}_{\eta\eta})'+2Du^{\epsilon}(Q^2)_{\eta}(Du^{\epsilon}_\eta)'+Du^{\epsilon}_\eta Q^2(Du^{\epsilon}_\eta)' \\\nonumber
&=F_{\eta\eta}-\frac{1}{2}Du^{\epsilon}(Q^2)_{\eta\eta}(Du^{\epsilon})'\leq C.
\end{align}
Compute the derivative of the \eqref{3.22}  w.r.t. $\eta$ again, {\bf{(H7)}} and Lemma \ref{Lem3.2} give that
\begin{align}\label{3.33}
&\quad -\partial_{t}u^{\epsilon}_{\eta\eta\eta}-\epsilon\Delta u^{\epsilon}_{\eta\eta\eta}-\mathbf{tr}(AD^{2}u^{\epsilon}_{\eta\eta\eta}+3A _{\eta}D^{2}u^{\epsilon}_{\eta\eta}+3A_{\eta\eta}Du^{\epsilon}_\eta+A_{\eta\eta\eta}D^{2}u^{\epsilon})\\\nonumber
&\quad+Du^{\epsilon}Q^2(Du^{\epsilon}_{\eta\eta\eta})'+3Du^{\epsilon}_\eta Q^2(Du^{\epsilon}_{\eta\eta})'+3Du^{\epsilon}(Q^2)_{\eta}(Du^{\epsilon}_{\eta\eta})'\\\nonumber
&=F_{\eta\eta\eta}-3Du^{\epsilon}_\eta(Q^2)_{\eta}(Du^{\epsilon}_\eta)'-3Du^{\epsilon}(Q^2)_{\eta\eta}(Du^{\epsilon}_{\eta})'
-\frac{1}{2}Du^{\epsilon}(Q^2)_{\eta\eta\eta}(Du^{\epsilon})'\leq C.
\end{align}
Now we construct the auxiliary function
\begin{equation}\label{A}
\bar{u}^{\epsilon}=\bar{u}^{\epsilon}(x,t,\eta):=u^{\epsilon}_{\eta\eta}+u^{\epsilon}_{\eta\eta\eta}.
\end{equation}
By calculations, we get
$$D_{(x,\eta)}\bar{u}^{\epsilon}=\big(Du^{\epsilon}_{\eta\eta}+Du^{\epsilon}_{\eta\eta\eta},2Du^{\epsilon}_{\eta}+3Du^{\epsilon}_{\eta\eta}\big),$$
and
$$D^2_{(x,\eta)}\bar{u}^{\epsilon}=\left(
                          \begin{array}{cc}
                            D^2u^{\epsilon}_{\eta\eta}+D^2u^{\epsilon}_{\eta\eta\eta} & 2D^2u^{\epsilon}_{\eta}+3D^2u^{\epsilon}_{\eta\eta} \\
                            2D^2u^{\epsilon}_{\eta}+3D^2u^{\epsilon}_{\eta\eta} & 2D^2u^{\epsilon}+6D^2u^{\epsilon}_\eta \\
                          \end{array}
                        \right).$$
Then we set the vector  $$ \bar{Q}=\big(Du^{\epsilon}Q^2,(Du^{\epsilon}Q^2)_\eta\big).$$
Recall $\hat{A}$ defined in \eqref{m1}, it follows from \eqref{3.22}, \eqref{3.33} and Lemma \ref{Lem3.2} that
\begin{align*}
-\partial_{t}\bar{u}^{\epsilon}-\epsilon\Delta \bar{u}^{\epsilon}-\mathbf{tr}(\hat{A}D^{2}\bar{u}^{\epsilon})+\bar{Q}(D\bar{u}^{\epsilon})'\leq C+\mathbf{tr}(A_{\eta\eta\eta}D^2\bar{u}^{\epsilon})+Du^{\epsilon}_\eta Q^2(Du^{\epsilon}_{\eta})'\leq C.
\end{align*}
Consider that $\|\bar{u}^\epsilon(\cdot,T)\|_{\infty}\leq C$ by the assumption {\bf{(H7)}} and we can conclude by comparison principle that $\|\bar{u}^\epsilon\|_{\infty}\leq C$ for a constant $C$ is independent of $\epsilon$. Since $u^\epsilon_{\eta\eta\eta}=\bar{u}^\epsilon-u^\epsilon_{\eta\eta}$ and $\|u^\epsilon_{\eta\eta}\|_{\infty}\leq C$ given in Lemma \ref{Lem3.2}, then we obtain $\|u^\epsilon_{\eta\eta\eta}\|_{\infty}\leq C$, it implies that the claim is valid.

\noindent{\bf{Step\ 2.}}
The claim \eqref{D3} gives that
$u_\eta^\epsilon$  is a sequence of uniformly semiconcave functions on $\mathbb{R}^{2}$, where $u^\epsilon_\eta$ satisfies \eqref{1}.
Lemma \ref{Lem3.2} gives that $u_\eta^\epsilon$ are uniformly bounded. Hence by Lemma \ref{2.1}, we know that $Du_\eta^\epsilon$ converges to $Du_\eta$ for a.e. $x\in\mathbb{R}^{2}$. By the arbitrary of $\eta$, we have ${D^{2}}u^\epsilon\rightarrow {D^{2}}u$ $a.e.$.

By {\bf(H2)} and {\bf(H3)}, let $\epsilon\rightarrow0$, we conclude that the viscosity solution $u$ satisfies the HJE \eqref{5.2} in the almost everywhere sense, which completes the proof.
\end{proof}

\textbf{Acknowledgments}

The authors are grateful to the referees for their careful reading and valuable comments.

\addcontentsline{toc}{section}{References} 
\end{document}